\documentclass[12pt, a4paper]{amsart}

\address{ }
\usepackage[left=3cm,top=3.5cm,right=3cm,marginparwidth=2.5cm]{geometry}
\usepackage{parskip, color}
\usepackage{geometry}

 \usepackage{hyperref}

\usepackage{marginnote}

\usepackage{tikz}
\newcommand{\MyScale}{4}
\usepackage{pdfpages}

\newcommand{\pd}[2]{\dfrac{\partial#1}{\partial#2}}
\DeclareMathOperator{\arcosh}{arcosh}

\newcommand{\R}{\mathbb{R}}

\newcommand{\So}{\mathbb{S}}
\newcommand{\Ho}{\mathbb{H}}
\newcommand{\x}{\lVert x \rVert}

\newcommand{\ve}{\varepsilon}

\newcommand{\pone}{\varphi_1}
\newcommand{\pIP}{\varphi_{\text{IP}}}
\newcommand{\sff}{\mathsf{f}}
\newcommand{\sfh}{\mathsf{h}}

\newcommand{\one}{A}
\newcommand{\two}{B}
\newcommand{\three}{C}
\newcommand{\four}{D}

\usepackage{cite}
\usepackage{setspace}

\theoremstyle{remark}

\newtheorem{theorem}{Theorem}[section]
\newtheorem{lemma}[theorem]{Lemma}
\newtheorem{prop}[theorem]{Proposition}
\newtheorem{cor}[theorem]{Corollary}

\begin{document}

\title[Fundamental gap of convex domains in $\mathbb H^n$]
{The vanishing of the fundamental gap of convex domains in $\mathbb H^n$ }

\author[T.~Bourni]{Theodora~Bourni}
\address{Department of Mathematics\\ University of Tennessee\\1403 Circle Dr, Knoxville, TN 37916 USA}
\email{tbourni@utk.edu}

\author[J.~Clutterbuck]{Julie~Clutterbuck}
\address{School of Mathematics \\
Monash University \\
9 Rainforest Walk,
VIC 3800 Australia}
\email{Julie.Clutterbuck@monash.edu}

\author[X.~H.~Nguyen]{Xuan~Hien~Nguyen}
\address{Department of Mathematics \\ Iowa State University \\ 411 Morrill Rd, Ames, IA 50011 USA}
\email{xhnguyen@iastate.edu}

\author[A.~Stancu]{Alina~Stancu}
\address{Department of Mathematics and Statistics \\ Concordia University\\ 1455 Blvd. de Maisonneuve Ouest \\ Montreal, QC, H3G 1M8 Canada}
\email{alina.stancu@concordia.ca}

\author[G.~Wei]{Guofang~Wei}
\address{Department of Mathematics \\ UC  Santa Barbara\\ Santa Barbara, CA 93106 USA}
\email{wei@math.ucsb.edu}

\author[V.~Wheeler]{Valentina-Mira~Wheeler}
\address{Valentina-Mira Wheeler \\ Institute for Mathematics and its Applications \\ University of Wollongong\\ Northfields Avenue\\ Wollongong, NSW 2522 Australia}
\email{vwheeler@uow.edu.au}

\thanks{The research of Julie Clutterbuck was supported by grant FT1301013 of the Australian Research Council. The research of Xuan Hien Nguyen was supported by grant 579756 of the Simons Foundation. The research of Alina Stancu was supported by NSERC Discovery Grant  RGPIN 327635. The research of Guofang Wei was supported by NSF Grant DMS 1811558. The research of Valentina-Mira Wheeler was supported by grants DP180100431 and DE190100379 of the Australian Research Council.}

\maketitle

\begin{abstract}
For the Laplace operator with Dirichlet boundary conditions on convex domains in $\mathbb H^n$, $n\geq 2$, we prove that the product of the fundamental gap with the square of the diameter can be arbitrarily small for domains of any diameter.
\end{abstract}



\section{Introduction}
We consider the low eigenvalues of the Laplace operator $- \Delta$ with Dirichlet boundary conditions on a convex, compact domain $\Omega$ of $\mathbb H^n$. This operator has a discrete spectrum with $\infty$ as its accumulation point. If the sequence of eigenvalues is arranged in increasing order $\lambda_1 < \lambda_2 \leq \lambda_3 \leq \cdots$, the \emph{fundamental gap} is the difference between the first two eigenvalues
  \[
     \lambda_2 - \lambda_1>0.
  \]

  A lot of work has been done in Euclidean and spherical spaces. In 2011, Andrews and Clutterbuck showed that on a convex domain in $\mathbb R^n$ with Dirichlet boundary condition, $\lambda_2 - \lambda_1 \ge 3\pi^2/D^2$, where $D$ is the diameter of the domain \cite{fundamental}.  The result is sharp, with the limiting case being rectangles that collapse to a line.    We refer to this paper for history and earlier work on this important subject, see also the survey article \cite{daifundamental}. More recently, Dai, He, Seto, Wang, and Wei (in various subsets) \cite{seto2019sharp,dai2018fundamental,he2017fundamental} generalized  the fundamental gap estimate to convex domains in $\mathbb S^n$, showing that the same bound holds:  $\lambda_2 - \lambda_1 \ge 3\pi^2/D^2$.

  Not much was known in the case of hyperbolic spaces. In a previous paper \cite{BCNSWW}, the authors first found an example of a convex domain in $\mathbb H^2$ in which the above lower bound is breached, thereby raising the question of estimating the fundamental gap for convex domains with small diameter. It is reasonable to believe that, as the diameters get smaller, the distortion from the metric becomes negligible and one would get a lower bound for the fundamental gap in terms of the diameter approaching $3 \pi^2/D^2$, the bound for Euclidean space, from below. The main result of this paper is the construction of explicit examples showing that, on the contrary, for any diameter, there is no lower bound on the gap.

 \begin{theorem}
	\label{thm:main}
	In hyperbolic spaces $\mathbb H^n$, $n\geq 2$, for any  constants $\ve>0$, $D>0$, there is a convex domain $\Omega$ with diameter $D$ whose fundamental gap satisfies
	\[
	\lambda_2(\Omega) - \lambda_1(\Omega) < \frac{\ve \pi^2}{D^2}.
	\]
\end{theorem}

From the discussion above, Theorem \ref{thm:main} shows that the behavior of the fundamental gap in hyperbolic spaces is drastically different from $\mathbb R^n$ and sphere cases. We explain the intuition behind the phenomenon in Section \ref{sec:discussion}.  We further remark that the quantity $D^2 (\lambda_2-\lambda_1)$ is invariant under the scaling of the metric. Hence the same result also holds for any simply connected negative constant curvature space forms.

For hyperbolic spaces, many explicit estimates on the upper and lower bounds of the first eigenvalue exist, see e.g. \cite{McKean1970, Gage1980, Savo2009, Artamoshin2016}. For the fundamental gap Benguria and Linde \cite{Benguria-Linde2007} obtained a beautiful upper bound for any open bounded domain $\Omega \subset \mathbb H^n$. Namely, the gap $\lambda_2(\Omega)- \lambda_1(\Omega) \le \lambda_2(B_\Omega)- \lambda_1(B_\Omega)$, where $B_\Omega$ is a ball in $\mathbb H^n$ such that $\lambda_1(B_\Omega) =  \lambda_1(\Omega)$. Our result gives new insight on the gap of convex domain in hyperbolic spaces.

Our work here draws strongly on work of Shih  \cite{shih1989counterexample}, who constructed a domain in $\mathbb{H}^2$ with a first eigenfunction that is not log-concave.  Shih's result highlights another difference from the situation in Euclidean cases, where the first Dirichlet eigenfunction is always log-concave \cite{MR0450480}.  Log concavity implies that the superlevel sets of the eigenfunction are convex.     We use domains similar to the ones in \cite{shih1989counterexample} and find that the first eigenfunction has two distinct maxima.     This means that the superlevel sets are very far from being convex:  they are not even connected.
 Thus, this article also gives a simpler proof of the existence of domains where the first eigenfunction is not log-concave.

 The organization of this paper is as follows.  We begin in dimension 2 for simplicity and because most of the insight can be garnered here. In Section \ref{sec:domain}, we explicitly construct the domain for the example, and describe its shape and diameter.  In Section \ref{sec:approach}, we sketch the main strategy.     In Section  \ref{sec:eigenvalues}, we make estimates on the first eigenvalue of the domain.    In Section \ref{sec:shape-of-h}, we describe precisely  the way in which the first eigenfunction is very small in the middle of the domain.   In Section \ref{sec:Rayleigh-quotient}, we show that the gap goes to zero.
In Section \ref{sec:discussion}, we give a heuristic explanation for the phenomenon with reference to a simple case in $\mathbb{R}^2$.  The generalization to higher dimensions is left until Section \ref{sec:higher-dimensions}.

\subsection*{Acknowledgements} This research originated at  the  workshop ``Women in Geometry 2" at the  Casa Matem\'atica Oaxaca (CMO)  from June 23 to June 28, 2019.  We would like to thank CMO-BIRS for creating the opportunity to start work on this problem through their support of the workshop.     Julie Clutterbuck thanks Ben Andrews for a useful discussion about the existence of a double-peaked eigenfunction.


\section{The domain}
\label{sec:domain}

Let $\mathbb H^2$ be the hyperbolic space modeled by the Poincar\'e half-plane $\{(x,y) \mid y>0\} = \{ (r, \varphi) \mid r>0, \varphi \in (-\pi/2, \pi/2)\}$ with the metric $g = ds^2 = \frac{dx^2 + dy^2}{y^2}$. Note that the coordinates $(r, \varphi)$ are not standard polar coordinates and are related to $(x,y)$ by $x= r \sin \varphi$, $y = r \cos \varphi$.

Let our domain be given by
  \[
  \Omega_{\sqrt{\mu}, L}:=\lbrace (r,\varphi):   1\le r\le e^{\pi/\sqrt{\mu}}, -L\le \varphi\le L \rbrace,
  \]
  where we start with an arbitrary fixed $L<\pi/2$, but will choose a suitable (large) positive $\mu$ later.

  For easier reference, we label some points of our domains (see Figure \ref{fig:OmegaDomain})
\begin{align*} P&=(\sin L, \cos L), & Q&=e^{\pi/\sqrt \mu}(\sin L, \cos L), & R&=e^{\pi/\sqrt \mu}(-\sin L, \cos L),  \\
  S&=(-\sin L, \cos L), & T &= (0,1), & U&=(0, e^{\pi/\sqrt \mu}).
\end{align*}

\begin{figure}[htbp]
    \begin{tikzpicture}[line cap=round,line join=round,scale=\MyScale]
    \fill[draw=blue, fill=blue!40]
    ({cos(30)}, {sin(30)})--({1.1*cos(30)},{1.1*sin(30)}) arc(30:150:1.1) -- ( {cos(150)}, {sin(150)}) arc(150:30:1);
\draw[thick,->] (-1.3,0)--(1.3,0); 
\draw[thick,->] (0,0)--(0,1.5); 
\draw [line width=2pt,color=blue] ({cos(30)}, {sin(30)}) node[below]{$P$}--({1.1*cos(30)},{1.1*sin(30)}) node[below, right]{$Q$}; 
\draw [line width=2pt,color=blue]( {cos(150)}, {sin(150)}) node[below]{$S$}--({1.1*cos(150)},{1.1*sin(150)}) node[below, left]{$R$}; 
\draw [line width=2pt,color=blue,domain=30:150] plot ({cos(\x)}, {sin(\x)});
\draw [line width=2pt,color=blue,domain=30:150] plot ({1.1*cos(\x)}, {1.1*sin(\x)});

\node[anchor=center] at (0.2,0.58) {\(L\)};
\draw[densely dashed, thick, ->] (0,0.5)  arc (90:30:0.5); 
\draw[densely dashed, thick, -] (0,0) --({1.1*cos(30)}, {1.1*sin(30)}); 

\node[anchor=east, blue] at (0, 0.9) {$T$};
\node[anchor=east, blue] at (0, 1.2) {$U$};
    \end{tikzpicture}%
  \caption{Domain $\Omega_{\sqrt{\mu}, L}= \{(r,\varphi) \mid 1 < r< e^{\pi/\sqrt{\mu}},\ -L< \varphi<L\}
$.}
\label{fig:OmegaDomain}
\end{figure}
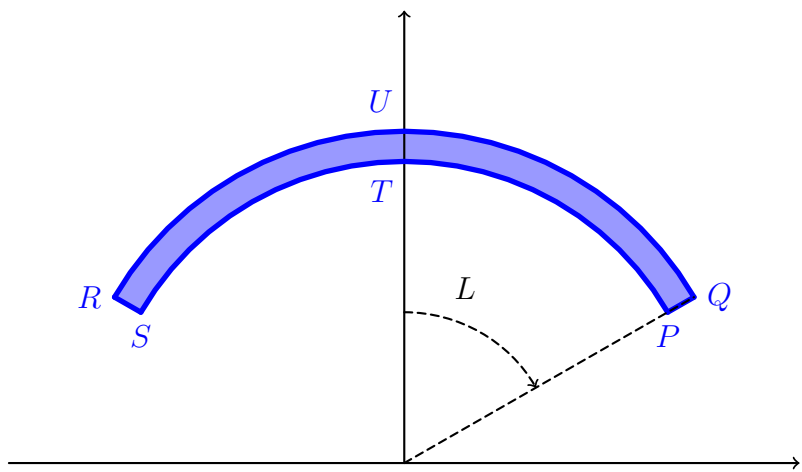

In an earlier paper  studying the fundamental gap \cite{BCNSWW}, the authors considered a similar domain $\Omega_{c, \theta_0, \theta_1}$ that may not be symmetric with respect to the geodesic $\varphi =0$. The domains of the two papers differ by a
 slight change of coordinates in which the new variables are:
	\[
	  \varphi = \frac{\pi}{2} - \theta, \quad L = \frac{\pi}{2} - \theta_{\ast}, \quad \mu=c^2,
	\]
	where $\theta_* = \min \{\theta_0, \pi - \theta_1\}$.

 Lemmas 3.3 and 4.3 in our earlier paper \cite{BCNSWW}  showed that the gap of $\Omega_{\sqrt \mu, L}$ goes to $\frac{3 \pi^2}{D^2}$ when $\mu$ is fixed and $L$ goes to zero. Those convex domains are thin strips along a segment of $y$-axis.  In this paper we will focus on the convex domains with $L$ fixed  and  $\mu \to \infty$, namely thin strips along part of the upper unit circle as in Figure \ref{fig:OmegaDomain}.

\subsection{The diameter}

From Proposition 4.1 of \cite{BCNSWW}, we have that the diameter $D_{\sqrt \mu, L}$ of $\Omega_{\sqrt \mu, L}$ is given by
\[D_{\sqrt \mu, L}= \max \{\textrm{dist}(P,Q), \textrm{dist}(P, R),  \textrm{dist}(R, S)\}. \]
Since this domain is symmetric  with respect to $\varphi=0$, we have $\textrm{dist}(R, S)=\textrm{dist}(P,Q)$.   Hence we conclude that:

\begin{lemma}
  On a domain $\Omega_{\sqrt \mu, L}$, the diameter is realized on the geodesic joining $P$ and $R$.
\end{lemma}
\begin{proof}
We recall that distances in hyperbolic half-plane Poincar\'e model  are given by
\begin{align}
    \text{dist}\left( (x_1,y_1),(x_2,y_2) \right) &=
    \label{eq:distance}
    \arcosh \left( \frac{(x_1^2+y_1^2) + (x_2^2+y_2^2) - 2x_1x_2}{2 y_1y_2}\right).
      \end{align}

Thus,
\begin{align*}
  \textrm{dist}(R, S)&= \arcosh\left( \frac{1+e^{2 \pi/\sqrt \mu}-2 e^{\pi/\sqrt{\mu}}(\sin L)^2}{2 e^{\pi/\sqrt{\mu}}(\cos L)^2}\right) \\
  \textrm{dist}(P,R)&= \arcosh\left( \frac{1+e^{2 \pi/\sqrt \mu}+2 e^{\pi/\sqrt{\mu}} (\sin L)^2}{2 e^{\pi/\sqrt{\mu}}(\cos L)^2}\right)
\end{align*}
and, since the argument of the latter is strictly greater than the argument of the former, and $\arcosh$ is increasing, the diameter must be realized on the geodesic joining $P$ and $R$.
\end{proof}

We emphasize that Figure \ref{fig:OmegaDomain} may be deceiving as in $\mathbb{H}^2$, the distance from $T$ to $U$ is smaller than the distance from $P$ to $Q$ (or  that from $R$ to $S$). Indeed, using the formula for distance  \eqref{eq:distance}, the inequalities $|\sinh x| \geq |x|$ and $\arcosh(x^2+1) \geq \sqrt 2 x$, we get
  \begin{align}
 \notag \textrm{dist}(R, S)
  & =  \arcosh\left( \frac{(e^{\pi/\sqrt \mu}-1)^2}{2e^{\pi/\sqrt \mu}} \frac{1}{(\cos L)^2} + 1 \right) \\
\notag & = \arcosh \left( 2 \left(\sinh \left( \tfrac{\pi}{2 \sqrt \mu} \right) \right)^2\frac{1}{(\cos L)^2} +1 \right)\\
\label{eq:neck-size}& \geq \frac{1}{\cos L} \frac{\pi}{\sqrt \mu} = \frac{1}{\cos L}\, \textrm{dist}(T,U).
\end{align}

Finally, we remark that, for all $\mu$ larger than a fixed constant $\mu_2$, we can bound the diameter in terms of $L$.

\begin{prop}
\label{prop:diameter-bounds}Given any  positive constant $\mu_2$, for all $\mu > \mu_2$,  the diameter $D_{\sqrt \mu, L}$ of $\Omega_{\sqrt \mu, L}$ is bounded by
  \begin{equation}
    \label{eq:diameter-bounds}
  \arcosh (1+2(\tan L)^2)  \leq D_{\sqrt \mu, L} \leq \arcosh (1+2(\tan L)^2) + \pi/\sqrt{\mu_2}.
  \end{equation}
\end{prop}

\begin{proof}
The diameter is bounded below by $\textrm{dist}(S,P)$:
\begin{align*}D_{\sqrt \mu, L}&\ge \textrm{dist}(S,P) = \arcosh\left( \frac{1+(\sin L)^2}{(\cos L)^2}\right)
= \arcosh\left(1+2\left(\tan L\right)^2 \right).
\end{align*}
For the bound from above, the distance formula \eqref{eq:distance} gives that $\textrm{dist}(U,R) = \textrm{dist}(T,S)$, so we have
  \begin{align*}
    D_{\sqrt \mu, L} &  \leq \textrm{dist}(P, T)+\textrm{dist}\left( T,U \right)+ \textrm{dist}\left( U, R \right)\\
      & = \textrm{dist}(P, S) + \textrm{dist}(T,U)\\
	& = \arcosh (1+2(\tan L)^2) + \pi/\sqrt{\mu}. \qedhere
      \end{align*}
    \end{proof}


\section{Sketch of the proof of Theorem \ref{thm:main} for $n=2$}
\label{sec:approach}
With the upper and lower bounds on the diameter from Proposition \ref{prop:diameter-bounds}, to prove Theorem \ref{thm:main}, it suffices to show that given $L \in (0, \pi)$, the fundamental gap $\lambda_2(\Omega_{\sqrt \mu, L}) - \lambda_1(\Omega_{\sqrt \mu, L}) \to 0$ as $\mu \to \infty$.

The domains $\Omega_{\sqrt{\mu}, L}$ were chosen because they allow for separation of variables \cite{shih1989counterexample, BCNSWW}.  The  eigenfunctions for the Laplace operator can be obtained by $u(r,\varphi)=h(\varphi)f(r)$, with
\begin{align} \notag r^2f_{rr}+ rf_r&=-\mu f \text{ on }r \in (1, e^{\pi/\sqrt{\mu}}) \\
 \label{eq:h-equation}
 h_{\varphi\varphi}+\lambda (\sec\varphi)^{2} h&=\mu h \text{ on }\varphi\in(-L,L),
\end{align}
where $\lambda$, $f$, and $h$ all depend on $\mu$ and where $f$ and $h$ satisfy Dirichlet boundary conditions. As pointed out in our earlier paper \cite[Section 2.3]{BCNSWW},  the first eigenvalue $\lambda_1$ of \eqref{eq:h-equation} is equal to  $\lambda_1(\Omega_{\sqrt \mu, L})$, and
the second eigenvalue $\lambda_2$ of \eqref{eq:h-equation} is not necessarily equal to $\lambda_2(\Omega_{\sqrt{\mu}, L})$\footnote{It is the case that $\lambda_2 = \lambda_2(\Omega_{\sqrt{\mu}, L})$ for $\mu$ large. Lemma \ref{lem:bound from below2} and Proposition \ref{prop:bound from above} show that for $\eta \in (0, L)$ and $\mu > \mu_2$, we have
	$
	  \lambda_1^{4\mu} \geq (\cos L)^2 4 \mu > (\cos \eta)^2 \mu \geq \lambda_2^{\mu},
	$
	where $\lambda_1^{4\mu}$ is the first eigenvalue of $ h_{\varphi\varphi}+\lambda (\sec\varphi)^{2} h=4\mu h$.}, the second eigenvalue of the Laplace operator on our domain $\Omega_{\sqrt \mu, L}$, but it is certainly true that
	\[
	  \lambda_2(\Omega_{\sqrt \mu, L}) \leq \lambda_2.
      \]
      As a consequence, $\lambda_2(\Omega_{\sqrt \mu, L}) - \lambda_1(\Omega_{\sqrt \mu, L}) \leq \lambda_2 - \lambda_1$. Therefore it suffices to show that $\lambda_2 - \lambda_1 \to 0$ as $\mu \to \infty$. A large part of this paper is concerned with studying eigenvalues and eigenfunctions of \eqref{eq:h-equation}.

\subsection*{The first eigenvalue}
Note that $\lambda_1$ is not bounded as $\mu \to \infty$. If it were, we would have that $h_{\varphi\varphi} = h (\mu - \lambda_1 (\cos \varphi)^{-2}) \geq 0$ for $\mu$ large when $h$ is the first (nonnegative) eigenfunction, which contradicts the fact that $h_1$ vanishes at the boundary. The first step is to capture how fast $\lambda_1$ grows as $\mu \to \infty$. This is done in Section \ref{sec:eigenvalues}.
\subsection*{Rayleigh quotients}
The first eigenvalue is a minimum of the Rayleigh quotient
  \begin{equation}
  \label{Rayleigh}
  R[h]:= \frac{\displaystyle\int_{-L}^L (h_{\varphi})^2 + \mu h^2 \,d\varphi}{\displaystyle\int_{-L}^L h^2 (\cos\varphi)^{-2} \,d\varphi},
  \end{equation}
  over  the Sobolev space $\mathcal H = \{ h \in W^{1,2}\left( (-L, L)\right)\mid h(-L)=h(L) = 0\} $.
  \begin{figure}[htbp]
  \begin{tikzpicture}[line cap=round, line join=round, scale=.3*\MyScale]
  \draw[thick, ->] (-2.8,0) -- (2.8,0) node[below] {$\varphi$};
  \draw[line width=2pt, domain=-1.7:1.7,smooth,variable=\x,blue!60] plot ({\x},{8*2.7^(-4)*cosh(16*\x/8)});
  \draw[line width=2pt, blue!60](1.7,2.26) .. controls (2.0,3.6) .. (2.3, 0);
  \draw[line width=2pt, blue!60](-1.7,2.26) .. controls (-2.0,3.6) .. (-2.3, 0);
  \draw[thick, dotted](-1, 0) -- (-1,.5);
  \node[below] at (-1, 0) {$-\varphi_0$};
  \draw[thick, dotted] (1, 0) -- (1,.5);
  \node [below] at (1, -.05) {$\varphi_0$};
  \node [below] at (-2.3,0) {$-L$};
  \node [below] at (2.3,0) {$L$};
    \end{tikzpicture}
    \caption{Expected graph of $h_1$. The maxima move towards $-L$ and $L$ and $h_1(0) \to 0$ as $\mu \to \infty$. The constant $\varphi_0 \in (0, L/2)$ is fixed and used in Section \ref{sec:shape-of-h}.
}
\label{fig:GraphH1}
\end{figure}
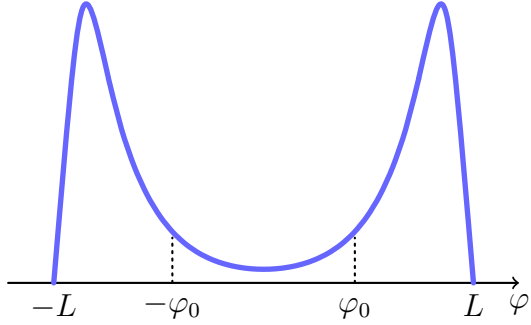

\subsection*{The first eigenfunction $h_1$}
A good grasp of $h_1$ is needed for estimating the Rayleigh quotient so in Section \ref{sec:shape-of-h}, we make precise the characteristics of the first eigenfunction: it has two maxima points as expected; as $\mu \to \infty$, the points where the maxima occur move towards $-L$ and $L$ respectively; the value $h_1(0)$ decays to zero and the first eigenfunction becomes more flat near $\varphi=0$.


\subsection*{Upper bound on the fundamental gap with Rayleigh quotients.}
For the first eigenvalue, we just take $\lambda_1=R[h_1]$. The second eigenvalue $\lambda_2$ is not computed directly, but bounded above by the Rayleigh quotient of an appropriate test function. The simplest way to obtain such a function is to multiply $h_1$ by the following odd function:  let $\psi(s)$ be the continuous piecewise linear function
  \[
    \psi =
      \begin{cases}
	1 \text{ for } \varphi < - \pone,\\
	- \varphi/\pone \text{ for } |\varphi| \leq \pone,\\
	-1 \text{ for } \varphi >  \pone,
      \end{cases}
  \]
  where $\pone = \frac{\varphi_0}{\mu}$ (see Figure \ref{fig:GraphPsi}).  The function $\psi h_1$ is odd and matches $h_1$ on $(-L_,-\pone)$ and $-h_1$ on $(\pone, L)$.

\begin{figure}[htbp]
  \begin{tikzpicture}[line cap=round,line join=round,scale=.3*\MyScale]
  \draw[thick, ->] (-2.6,0) -- (2.6,0) node[below] {$\varphi$};
  \node[above] at (-2,1) {$\psi = 1$};
  \node[below] at (2,-1) {$\psi =-1$};
  \draw[thick, dotted](-.1, -1.2) -- (-.1,1.2);
  \node[below] at (-.3, -1.2) {$-\pone$};
  \draw[thick, dotted] (.1, -1.2) -- (.1,1.2);
  \node [below] at (.25, -1.25) {$\pone$};
\draw [line width=2pt] (-2.5,1)--(-.1,1)--(.1,-1)--(2.5,-1);
    \end{tikzpicture}%
    \caption{Graph of $\psi$.}
\label{fig:GraphPsi}
\end{figure}
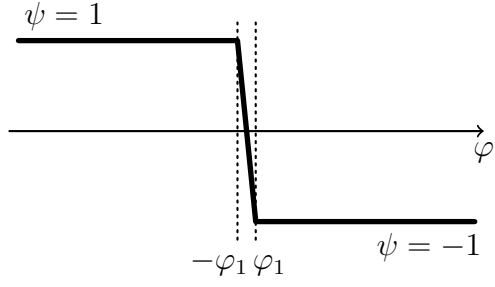

As mentioned above, we infer that
  \begin{equation}
    \lambda_2 - \lambda_1 \leq R[\psi h_1] - R[h_1].  \label{gap-est}
  \end{equation}
  Hence for estimating the fundamental gap from above,  it suffices to find an upper bound on the right-hand side of the inequality consisting of the quotients' difference. The difference is concentrated on the interval $(-\pone, \pone)$, where we have that $h_1$ and its derivatives are small. The computation is done in Section \ref{sec:Rayleigh-quotient} using estimates on $h_1$ from Section \ref{sec:shape-of-h}.

\section{The first eigenvalue}
\label{sec:eigenvalues}

Before we can prove that the first eigenfunction has the shape given in Figure \ref{fig:GraphH1}, we need estimates on the first eigenvalue.

Recall the equation for the eigenfunctions
  \begin{equation}
  \tag{\ref{eq:h-equation}}
  h_{\varphi\varphi} + (\lambda (\sec \varphi)^{2} - \mu) h = 0.
  \end{equation}
  We emphasize that the first eigenvalue $\lambda_1 = \lambda_1^{\mu}$ and the corresponding nonnegative eigenfunction $h_1=h_1^{\mu}$ both depend on $\mu$, even though it will not always be showcased in the notation.

  The bound from below is a straightforward application of Wirtinger inequality. We will use the weaker $\lambda_1 \geq (\cos L)^2 \mu$ in the rest of the article but include the stronger estimate for completeness.
Note that Lemma \ref{lem:bound from below2} gives us that $\lambda_1 \to \infty$ as $\mu \to \infty$.

  \begin{lemma}[Bound from below] \label{lem:bound from below2} $\lambda_1\ge (\cos L)^2 \left( \frac{\pi^2}{4L^2}+ \mu\right).$
  \end{lemma}
  \begin{proof}
    From the characterization of the first eigenvalue through the Rayleigh quotient \eqref{Rayleigh} on the Sobolev space $\mathcal H = \{ h \in W^{1,2}\left( (-L, L)\right)\mid h(-L)=h(L) = 0\} $, we have
  \begin{align*}
    \lambda_1^{\mu}&= \inf_{h\in \mathcal{H}}    \frac{\displaystyle\int_{-L}^{L} \left( {h}_\varphi   \right)^2  + \mu {h}^2 \,d\varphi}{\displaystyle\int_{-L}^{L} {h}^2  (\sec\varphi)^2 \,d\varphi} \ge  \inf_{h\in \mathcal{H}}    \frac{\displaystyle\int_{-L}^{L}  \left(\tfrac{\pi}{2L}\right)^{2}  h^2+  \mu {h}^2 \,d\varphi} {\displaystyle\int_{-L}^{L} {h}^2  (\sec\varphi)^2 \,d\varphi}\\
&\ge  \inf_{h\in \mathcal{H}}   \frac{ \left(\left(\tfrac{\pi}{2L}\right)^{2}+\mu\right) \displaystyle\int_{-L}^{L}    h^2 \,d\varphi} {(\sec L)^2\displaystyle\int_{-L}^{L} {h}^2  \,d\varphi}
=  (\cos L)^2\left(\left(\tfrac{\pi}{2L}\right)^{2}+\mu\right). \qedhere
  \end{align*}
  \end{proof}

  We now control the rate of growth $\lambda_1 = \lambda_1^{\mu}$ from above.

   \begin{prop}\label{prop:bound from above}
    For every $\eta \in (0,L)$, there exists a $\mu_2$ such that $\mu>\mu_2$ implies
      \begin{equation}
	\label{mu-bound-on-lambda}
	\mu (\cos \eta)^2 \geq \lambda_1.
      \end{equation}
  \end{prop}

  \begin{proof}
    We argue by contradiction and assume that there is an $\eta \in (0, L)$ and a sequence $\mu_k \to \infty$ so that
    \[
      \frac{\lambda_1^{\mu_k}}{\mu_k} \geq (\cos \eta)^2.
    \]

    For those $\mu$'s and corresponding $\lambda$'s,
    \begin{align*}
      \lambda_1^{\mu_k} (\cos \varphi)^{-2}  - \mu_k
      & \geq  \lambda_1^{\mu_k} \left((\cos \varphi)^{-2} - (\cos \eta)^{-2}\right).
    \end{align*}
    For $\varphi \in \left(-L, -\tfrac{L+\eta}{2}\right)$, the coefficient of $\lambda_1^{\mu_k}$ is positive, bounded away from zero. Taking $\mu_k$ and therefore $\lambda_1^{\mu_k}$ large enough, we can make the right-hand side larger than $\tfrac{\pi^2}{4(L-\eta)^2}$. For these large $\mu_k$'s, Sturm's Comparison Theorem applied to \eqref{eq:h-equation} and $h_{\varphi\varphi} + \tfrac{\pi^2}{4(L-\eta)^2} h =0$ would imply that $h_1^{\mu_k}$ has a zero in $\left(-L,-\tfrac{L+\eta}{2}\right)$. This contradicts the fact that the first eigenfunction is positive in $(-L, L)$.
  \end{proof}

  Combining Lemma \ref{lem:bound from below2} and Proposition \ref{prop:bound from above}, we get the following asymptotic behavior for the ratio of $\lambda_1$ and $\mu$:
  \begin{cor}
    \label{cor:quotient-lambda-mu}
    $\frac{\lambda_1}{\mu} \searrow (\cos L)^2$ as $\mu \to \infty$.
    \end{cor}

\section{The shape of $h_1$}
\label{sec:shape-of-h}

In this section, we show that the first eigenfunction behaves as claimed in Figure \ref{fig:GraphH1} and obtain estimates for the rate at which $h_1 (0)$ tends to zero. This is done first by an integral estimate, then a pointwise estimate, then an improved integral estimate, then finally an integral estimate on the derivative.

  The first eigenfunction of \eqref{eq:h-equation} is even because all the coefficients of \eqref{eq:h-equation} are even. We also assume that $h_1 >0$ on $(-L,L)$ and is normalized so that $\int_{-L}^L h_1^2 (\sec \varphi)^2 d\varphi =1$.

  The first eigenfunction has two inflection points $\pm \pIP$ situated where $(\cos \pIP)^2 = \lambda_1/\mu$. From Corollary \ref{cor:quotient-lambda-mu}, we know that those inflection points exist (i.e. the equation is satisfied in $(-L,L)$) and that $\pIP \to L$ as $\mu \to \infty$. Going to Proposition \ref{prop:bound from above}, we can describe the behavior a little better:
  \begin{cor}
    \label{cor:monotonicity}
    For any $\eta \in (0, L)$, there is a positive constant $\mu_2$ so that the inflection points for $h_1^{\mu}$ are outside of the interval $(-\eta, \eta)$ whenever $\mu > \mu_2$. Therefore the maxima of $h_1$ are at points outside of the interval $(-\eta, \eta)$ and the function $h_1$ is increasing on $(0, \eta)$.
  \end{cor}

  The last property is a consequence of the concavity of $h_1$ and the fact that $h_1'(0) =0$. If we seek integral bounds on some fixed interval $(-\varphi_0,\varphi_0)$, we can take $\eta \in (\varphi_0, L)$. With the corresponding $\mu_2$ from Proposition \ref{prop:bound from above}, we have that $h_1^{\mu}$ is increasing on $(0, \varphi_0)$ for all $\mu > \mu_2$. 

\subsection{A uniform integral bound on a subinterval}

\begin{lemma}
  \label{lem:integral-bound-h1}
  Given $\varphi_0 \in (0, L)$, there exists a function $b(\mu, \varphi_0)$ with $b(\mu, \varphi_0) \to 0$ as $\mu \to \infty$ so that
  \begin{equation}
    \label{eq:l2-bound-on-subinterval}
    \int_{-\varphi_0}^{\varphi_0} h_1^2 (\sec \varphi)^2  d\varphi < b(\mu, \varphi_0),
  \end{equation}
      where the first eigenfunction is normalized so that $\int_{-L}^L   h_1^2 (\sec \varphi)^2  d\varphi =1$.
\end{lemma}

\begin{proof}
  We first give an estimate of $\lambda_1$ from below. By \eqref{Rayleigh} and the normalization of $h_1$, we have
  \begin{align*}
    \lambda_1 &=R[h_1]  > \mu \int_{-L}^L h_1^2 d\varphi \\
    & \geq 2 \mu (\cos L)^2 \int_{\varphi_0}^L h_1^2 (\sec \varphi)^2 d\varphi + 2 \mu (\cos \varphi_0)^2 \int_0^{\varphi_0} h_1^2 (\sec \varphi)^2 d\varphi\\
    & = \mu (\cos L)^2 + \mu \left( (\cos \varphi_0)^2 - (\cos L)^2 \right) \int_{-\varphi_0}^{\varphi_0} h_1^2 (\sec \varphi)^2 d\varphi.
  \end{align*}
 Hence \begin{equation*}
    \frac{ \lambda_1}{\mu}  -(\cos L)^2 > \left( (\cos \varphi_0)^2 - (\cos L)^2 \right) \int_{-\varphi_0}^{\varphi_0} h_1^2 (\sec \varphi)^2 d\varphi.
  \end{equation*}
  Simply set $b(\mu, \varphi_0)= \frac{\lambda_1 / \mu - (\cos L)^2 }{ (\cos \varphi_0)^2 - (\cos L)^2 }$. The fact that $b(\mu, \varphi_0) \to 0$ as $\mu \to \infty$ follows from Corollary \ref{cor:quotient-lambda-mu}.
\end{proof}

\subsection{A pointwise lower and upper bound on $h_1$ near $0$}
\label{ssec:bound-h(0)}

From now on, $\varphi_0$ is a constant in $(0, L/2)$.

We use the following Sturm comparison for Jacobi equations to obtain a lower bound for the first eigenfunction  $h_1$ near $0$.

  \begin{theorem}[Sturm Comparison Theorem]   For $i=1,2$, let $f_i$ satisfy
$${f_i}'' + b_i f_i=0 \text{ on }(0,l),$$
and $f_1(0)=f_2(0)>0$,  $f'_i(0)=0$.     Suppose that $b_1\ge b_2$ and $f_1 >0$ on $(0,l)$.   Then $f_1\le f_2$ on $(0,l)$.
If $f_1=f_2$ at $t_1 \in (0,l)$, then $b_1 \equiv b_2$ on $(0,t_1)$.
  \end{theorem}
  \begin{proof}
    The theorem is well known. For example, it is stated in \cite[Page 238-239]{doCarmo}  for the initial conditions  $f_1(0)=f_2(0)=0$,  $f'_1(0)= f_2'(0) >0$. Clearly, the same proof works for the above dual initial conditions.
\end{proof}

 Recall once again that $h_1$ satisfies the Jacobi equation:
 \begin{equation}
  \tag{\ref{eq:h-equation}}
  h_{\varphi\varphi} + (\lambda (\sec \varphi)^{2} - \mu) h = 0.
  \end{equation}with $h_1'(0) =0$.
Since,   by Proposition \ref{prop:bound from above}, $\lambda_1 \le \mu(\cos(2\varphi_0))^2$ for all $\mu$ sufficiently large, and on $(-\varphi_0,\varphi_0)$, $(\sec\varphi)^2 \le (\sec \varphi_0 )^2$, we have
\[
  \lambda_1 (\sec\varphi)^2-\mu \le  \mu\frac{(\cos(2\varphi_0))^2}{(\cos\varphi_0)^2} - \mu
=-c_1\mu,
\]
where we set $c_1= 1 -\left( \frac{\cos(2 \varphi_0)}{\cos \varphi_0)} \right)^2$. Remark that $c_1>0$.

Let $\bar{h}(\varphi) = h_1(0)\cosh(\sqrt{\mu c_1}\varphi)$.  Then $\bar{h}$ satisfies the Jacobi equation \[
 h_{\varphi\varphi} - c_1 \mu h =0\]
with $\bar{h}'(0) =0$ and $\bar{h}(0) = h_1(0).$

By the Sturm Comparison Theorem above, we have $h_1(\varphi)\ge \bar{h}(\varphi)$ on $[0, \varphi_0)$, thereby on $(-\varphi_0, \varphi_0)$ because both functions are even.

We formulate this as a lemma below.
\begin{lemma}   Fix $\varphi_0  \in (0,L/2)$.  Then for all $\mu$ sufficiently large,
	\begin{equation}
	\label{eq:h0 from above}
	h_1(\varphi)\ge h_1(0)\cosh(\sqrt{\mu c_1}\varphi),
	\end{equation}
	for all $|\varphi|<\varphi_0$. Here $c_1= 1 -\left( \frac{\cos(2 \varphi_0)}{\cos \varphi_0} \right)^2$.
\end{lemma}

Similarly, using the lower bound for $\lambda_1$ (Lemma~\ref{lem:bound from below2}) and $(\sec\varphi)^2 \ge 1$, we have
\[
\lambda_1 (\sec\varphi)^2-\mu \ge  (\cos L)^2 \left( \frac{\pi^2}{4L^2} +\mu \right)  - \mu \ge \mu \left[(\cos L)^2 -1 \right]
=- (\sin L)^2\mu,
\]
therefore
  \begin{equation}
  \label{eq:h0 from below}
h_1(\varphi) \le h_1(0) \cosh (\sqrt{\mu} \sin L \varphi).
  \end{equation}
  This last estimate is used to improve the integral bound. In the meantime, \eqref{eq:h0 from above} allows us to estimate $h_1(0)$.

\begin{lemma}
  Fix $\varphi_0  \in (0,L/2)$. Then for all $\mu$ sufficiently large depending on $\varphi_0$ and $L$, we have
  that
  \begin{equation}
    \label{eq:h0-bound}
    h_1(0)^2 \leq 4 b(\mu, \varphi_0) \exp(-\sqrt{\mu c_1} \varphi_0/2) \le C  \exp(-\sqrt{\mu c_1} \varphi_0/2),
  \end{equation}
  where $c_1 = 1 - \left(\frac{\cos (2 \varphi_0)}{\cos \varphi_0}\right)$, $b(\mu, \varphi_0)$ is the function from Lemma \ref{lem:integral-bound-h1}, and $C$ is a positive constant independent of $\varphi_0$ and $\mu$.
\end{lemma}

\label{b-discussion}
From the explicit form of $b(\mu, \varphi_0)$ in the proof of Lemma \ref{lem:integral-bound-h1}, we see that $b(\mu, \varphi_0)$ is bounded above uniformly for $\varphi_0 \in (0, L/2)$.

\begin{proof}
  Inserting \eqref{eq:h0 from above} into \eqref{eq:l2-bound-on-subinterval}, we obtain a rough estimate on $h(0)$ for $\mu$ large as follows
  \begin{align*}
    b(\mu, s_0) &
    \geq \int_{-\varphi_0}^{\varphi_0} h_1^2 (\sec \varphi)^2 d\varphi \geq 2\int_{0}^{\varphi_0} h_1^2 d\varphi \geq 2 \int_{0}^{\varphi_0} h_1(0)^2 (\cosh (\sqrt{\mu c_1} \varphi))^2 d\varphi \\
    & \geq \frac{1}{2}  \int_{0}^{\varphi_0} h_1(0)^2 \exp({2\sqrt{\mu c_1} \varphi}) d\varphi  = \frac{1}{4} h_1(0)^2 \frac{\exp(2 \sqrt{\mu c_1} \varphi_0)-1}{\sqrt{\mu c_1}}\\
    & \geq
    \frac{1}{4} h_1(0)^2 \frac{\exp(\sqrt{\mu c_1} \varphi_0)}{\sqrt{\mu c_1}},
   \end{align*}
   where we have used, in the first line, the even property of $h_1$. 
%
  The smaller coefficient for $\varphi_0$ in the exponential in \eqref{eq:h0-bound} is to compensate for the powers of $\mu$ outside of the exponential.  \end{proof}

\subsection{An improved integral bound}

\begin{lemma} For $\varphi_1 := \varphi_0/\mu$, we have
	\begin{align}
	\label{lastpartofI}
	\lim_{\mu \to \infty} \mu^2 \int_{-\varphi_1}^{\varphi_1} h_1^2 d \varphi = 0.
	\end{align}
\end{lemma}
\begin{proof}	
	Using now the upperbound \eqref{eq:h0-bound} on $h_1(0)$ in \eqref{eq:h0 from below} we obtain for $|\varphi| < \varphi_1$
	\begin{align*}
	h_1^2(\varphi)  \leq C  \exp(-\sqrt{\mu c_1} \varphi_0/2) \cosh^2 (\sqrt{\mu} \sin L \varphi).
	\end{align*}
	Then
	\begin{align*}
	\mu^2 \int_{-\varphi_1}^{\varphi_1} h_1^2 d \varphi & \leq
	2 \mu^2 C  \exp(-\sqrt{\mu c_1} \varphi_0/2)  \int_{0}^{\varphi_1} \cosh^2 (\sqrt{\mu} \sin L \varphi) d \varphi \\
	& =  C \mu^2 \exp(-\sqrt{\mu c_1} \varphi_0/2) \left[ \frac{\varphi_0}{\mu} +   \frac{e^{2 \sin L \frac{\varphi_0}{\sqrt{\mu}}}-e^{-2 \sin L \frac{\varphi_0}{\sqrt{\mu}}}}{4 \sqrt{\mu} \sin L} \right].
	\end{align*}
Since $C$ is independent of $\mu$, the last term goes to zero as $\mu \to \infty$.
\end{proof}


\subsection{An integral estimate on the derivative of $h_1$}
Using the bound for the first eigenvalue in Lemma~\ref{lem:bound from below2} we obtain a bound on $h_1'$ for $\varphi$ small.

\begin{lemma}
  For $\pone :=\varphi_0/\mu$, we have
    \begin{equation}
      \label{eq:integral-derivative-estimate}
      \lim_{\mu \to \infty} \int_{-\pone}^{\pone} (h_1')^2 d\varphi =0.
    \end{equation}
  \end{lemma}
  \begin{proof}
    Choose $\mu_2$ large  such that $h_1$ is increasing and convex on $(0, \varphi_0)$ for $\mu > \mu_2$ (see Corollary \ref{cor:monotonicity}). We first estimate $h_1'(\varphi)$ for $\varphi \in (0, \pone)$. From equation \eqref{eq:h-equation} and the fact that $h_1'(0)=0$, we have
  \begin{align*}
      h_1'(\varphi) & = \int_{0}^{\varphi} h_1''(t) dt = \int_0^{\varphi} \left( \mu - \lambda_1 (\sec t)^2 \right) h_1(t) dt\\
      & \leq \int_0^{\varphi} \left( \mu - \mu (\cos L)^2 (\sec t)^2  \right)h_1(t) dt  \leq \mu (\sin L)^2 \int_0^{\varphi} h_1(t)dt,
  \end{align*}
  where we used the lower bound on $\lambda_1$ from Lemma \ref{lem:bound from below2}. Since  $h_1'$ is increasing on $(0, \pone)$, the Cauchy-Schwartz inequality then implies
  \[
    \left(h_1'(\varphi_1)\right)^2 \leq  \mu^2 (\sin L)^4 \varphi_1 \int_0^{\pone} h_1^2 d\varphi.
  \]
  And the right-hand side goes to zero as $\mu \to \infty$ by \eqref{lastpartofI}.
  \end{proof}


\section{Estimating the Rayleigh quotient difference}
\label{sec:Rayleigh-quotient}

In the beginning of Section \ref{sec:approach}, we argued that Theorem \ref{thm:main} is a corollary of the bounds on the diameter \eqref{eq:diameter-bounds} and the following proposition.

\begin{theorem}
  \label{thm:FG-h-equation}
Given the equation
  \begin{equation}
  \tag{\ref{eq:h-equation}}
  h_{\varphi\varphi} + (\lambda (\sec \varphi)^{2} - \mu) h = 0, \quad \varphi \in (-L,L)
  \end{equation}
  with zero Dirichlet boundary conditions, the difference between the first and second eigenvalues satisfies
  \[
    \lambda_2 - \lambda_1 \to 0, \textrm{ as } \mu \to \infty.
  \]
\end{theorem}

\begin{proof}
  From inequality \eqref{gap-est} in Section \ref{sec:approach}, it suffices to show that $R[\psi h_1] - R[h_1] \to 0$ as $\mu \to \infty$ where $\psi$ is defined in Section \ref{sec:approach}. The difference between $|\psi h_1|$ and $|h_1|$ is supported on the interval $(-\pone, \pone)$ (see Figure \ref{fig:GraphH2}).

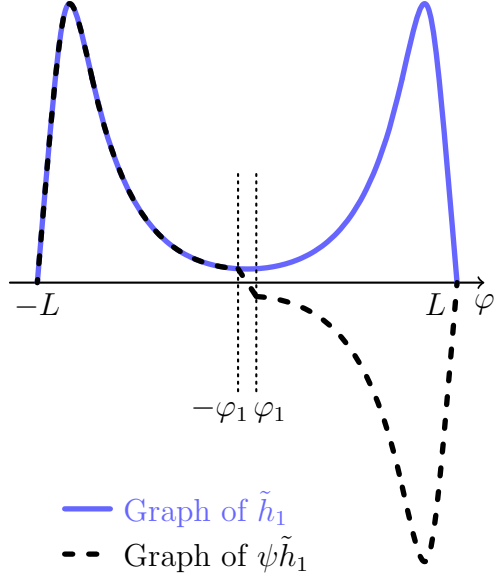
\begin{figure}[htbp]
  \begin{tikzpicture}[line cap=round, line join=round, scale=.3*\MyScale]
  \draw[thick, ->] (-2.6,0) -- (2.6,0) node[below] {$\varphi$};
  \draw[line width=2pt, domain=-1.7:1.7,smooth,variable=\x,blue!60] plot ({\x},{8*2.7^(-4)*cosh(16*\x/8)});
  \draw[line width=2pt, blue!60](1.7,2.26) .. controls (2.0,3.6) .. (2.3, 0);
  \draw[line width=2pt, blue!60](-1.7,2.26) .. controls (-2.0,3.6) .. (-2.3, 0);
  \draw[line width=2pt, domain=-1.7:-.1,loosely dashed,variable=\x,black] plot ({\x},{8*2.7^(-4)*cosh(16*\x/8)});
  \draw[line width=2pt, domain=-.1:.1,loosely  dashed,variable=\x,black] plot ({\x},{\x*(-80)*2.7^(-4)*cosh(16*\x/8)});
  \draw[line width=2pt, domain=.1:1.7,loosely  dashed,variable=\x,black] plot ({\x},{-8*2.7^(-4)*cosh(16*\x/8)});
  \draw[line width=2pt,loosely  dashed, black](1.7,-2.26) .. controls (2.0,-3.6) .. (2.3, 0);
  \draw[line width=2pt,loosely  dashed, black](-1.7,2.26) .. controls (-2.0,3.6) .. (-2.3, 0);
  \draw[thick, dotted](-.1, -1.2) -- (-.1,1.2);
  \node[below] at (-.3, -1.15) {$-\pone$};
  \draw[thick, dotted] (.1, -1.2) -- (.1,1.2);
  \node [below] at (.25, -1.2) {$\pone$};
  \node [below] at (-2.3,0) {$-L$};
  \node [below left] at (2.3,0) {$L$};
  \draw[line width=2pt, blue!60] (-2,-2.5) -- (-1.5,-2.5) node[right]{Graph of $\tilde h_1$};
  \draw[line width=2pt,loosely  dashed, black] (-2, -3) -- (-1.5, -3.0) node[right] {Graph of $\psi \tilde h_1$};
    \end{tikzpicture}
    \caption{Graphs of $h_1$ and $\psi h_1$.}
\label{fig:GraphH2}
\end{figure}

Before we start, we set the notation for the denominator of $R[\psi h_1]$
  \[
    \int_{-L}^L (\psi h_1)^2 (\sec \varphi)^2 d\varphi = 1 - \int_{-L}^L \left((h_1)^2 - (\psi h_1)^2 \right) (\sec \varphi)^2 d\varphi  = 1 - \text{\one},
  \]
  where $\text{\one} := \int_{-\pone}^{\pone}  \left((h_1)^2 - (\psi h_1)^2 \right) (\sec \varphi)^2 d\varphi$.

  The difference of the Rayleigh quotients is then
  \begin{align*}
    R[\psi h_1] - R[h_1] &= \frac{1}{1 - \text{\one}} \left\{ (1-\text{\one}) R[\psi h_1] - R[h_1] + R[h_1]\text{\one} \right\}\\
    & = \frac{1}{1-\text{\one}}\left( \text{\two} + \text{\three} + \text{\four} \right)
  \end{align*}
    where
  \begin{align*}
    \text{\two} &= \int_{-L}^{L} \left[ (\psi h_1)'^2 - h_1'^2\right] d\varphi =\int_{-\pone}^{\pone} \left[ (\psi h_1)'^2 - h_1'^2\right] d\varphi,\\
    \text{\three}&= \int_{-L}^{L} \mu\left[(\psi h_1)^2 - h_1^2 \right] d\varphi = \int_{-\pone}^{\pone} \mu\left[(\psi h_1)^2 - h_1^2 \right] d\varphi,\\
    \text{\four}&= \lambda_1\text{\one},
  \end{align*}
  because $R[h_1] = \lambda_1$. Recall that $\varphi_0 \in (0, L/2)$ and $\varphi_1 = \varphi_0/\mu$. We finish the proof by showing that A, B, C, and D  all
   go to zero as $\mu$ goes to infinity. Note that $0 <\text{\one} \leq \text{\four}$ so we can skip $\text{\one}$.

  We have
  \begin{align*}
   | \text{\two}| & \leq \int_{-\pone}^{\pone} \left(  h_1 \psi' +h_1' \psi \right)^2  d\varphi  +   \int_{-\pone}^{\pone} h_1'^2 d\varphi  \\
   &  \leq 2  \int_{-\pone}^{\pone} \left( h_1^2(\psi')^2 + (h_1')^2 \psi^2 \right)d\varphi +   \int_{-\pone}^{\pone} h_1'^2 d\varphi \\
    & \leq  2 \frac{\mu^2}{\varphi_0^2}   \int_{-\pone}^{\pone} h_1^2 d\varphi +3   \int_{-\pone}^{\pone} (h_1')^2 d\varphi.
  \end{align*}
  Both terms go to zero as $\mu \to \infty$ by \eqref{lastpartofI} and \eqref{eq:integral-derivative-estimate} respectively.

  For $\text{\three}$, note that
  \begin{align*}
   | \text{\three}|  \leq  \mu  \int_{-\pone}^{\pone}  h_1^2 d\varphi
    \end{align*}
    which tends to zero by \eqref{lastpartofI}.

    Finally, $\lambda_1 \leq \mu (\cos (L/2))^2$ for $\mu$ sufficiently large by Proposition \ref{prop:bound from above}. For such $\mu$'s, we have
  \begin{align*}
  0<  \text{\four} & \leq  \mu (\cos (L/2))^2 \int_{-\pone}^{\pone} h_1^2 (\sec \varphi)^2 d\varphi \leq  \mu \int_{-\pone}^{\pone} h_1^2 d\varphi,
  \end{align*}
  which tends to zero for $\mu \to \infty$ by \eqref{lastpartofI}. This completes the proof.
\end{proof}


\section{Heuristic argument}
\label{sec:discussion}

For a bounded, connected domain, the first eigenvalue is simple, and so the fundamental gap is always positive.    In order to understand our example, we begin by describing a simple situation in $\R^2$ where the first eigenvalue is not simple.


\begin{figure}[htbp]
\begin{tikzpicture}[line cap=round,line join=round, scale=.1*\MyScale]
\filldraw[draw=blue, fill=blue!40]
    ({-4-3*cos(186)},{3*sin(186)}) -- ({4+3*cos(186)},{3*sin(186)}) arc[start angle=186, end angle=534, radius=3] --({-4+3*cos(6)},{3*sin(6)}) arc[start angle=6, end angle=354, radius=3];
\draw [line width=2pt,color=blue] ({-4+3*cos(6)},{3*sin(6)}) arc (6:354:3);
\draw [line width=2pt,color=blue] ({4+3*cos(186)},{3*sin(186)}) arc (186:534:3);
\draw [line width=2pt, color=blue] ({-4+3*cos(6)},{3*sin(6)}) -- ({4-3*cos(6)},{3*sin(6)});
\draw [line width=2pt, color=blue] ({-4-3*cos(186)},{3*sin(186)}) -- ({4+3*cos(186)},{3*sin(186)});
\draw [line width=1pt,dash pattern=on 3pt off 3pt] (-4,0) circle (2.8);
\draw [line width=1pt,dash pattern=on 3pt off 3pt] (4,0) circle (2.8);
\end{tikzpicture}
\caption{The fundamental gap can be small if the domain has a neck.}
\end{figure}
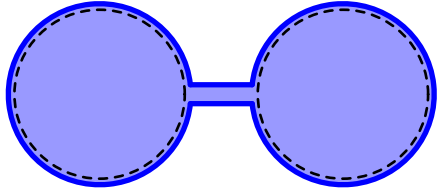

 Let $U$ be the disjoint union of two unit balls;    this domain is not convex and not connected.   The Dirichlet eigenfunctions of $U$ are given by combinations of the eigenfunctions on each ball, which are given by Bessel functions.  Let $u_i$ be the $i^{\text{th}}$ eigenfunction on the ball.    Let $\mu_i(B)$ be the $i^\text{th}$ eigenvalue on the ball.     Then the first eigenfunction for $U$ is given by two copies of $u_1$, translated to each ball.   The first eigenvalue of $U$ is $\mu_1(B)$.  The \emph{second} eigenfunction is given by a copy of $u_1$ on one ball, and a copy of $-u_1$ on the other ball.    We can see this is orthogonal to the first eigenfunction, but has the same eigenvalue:  the fundamental gap is zero.

The eigenvalues are continuous under perturbations of the domain.   Specifically, if we join the two components of $U$ by a small tube of width $\epsilon$ to create a new domain $U_\epsilon$, then $\lambda_k(U_\epsilon)\to \lambda_k(U)$ as $\epsilon\to 0$ \cite[Th 2.3.20]{henrot-book}.
On such a domain, the second eigenfunction is very close to $u_1$ on the first ball and $-u_1$ on the second ball.    In the neck joining the balls, the first and second eigenfunctions are  very small, and thus contribute very little to the Rayleigh quotient for either the first or second eigenvalue. Therefore the eigenvalues are $\epsilon$-close to those on $U$, and the fundamental gap is close to zero.


 In the case of convex domains in $\R^n$ and in $\So^n$, such dumbbell-shaped domains are excluded:  they are not convex.   However, in hyperbolic space, geodesics diverge, and thus we can find a convex domain with a narrow region separating regions of relatively large area.    These domains support eigenfunctions similar to that on the dumbbell domain described above, and therefore have very small gap.

 \begin{minipage}{0.45\textwidth}
 \begin{center}
 \begin{tikzpicture}[line cap=round,line join=round]

\draw [line width=2pt] (0,0) circle (2);
\draw [line width=2pt, dash pattern=on 3pt off 3pt] (-1,-1.73) arc(-30:50:2.38);
\draw [line width=2pt, dash pattern=on 3pt off 3pt] (1.88,0.68) arc(110:220.24:1.39);
\draw [line width=2pt, dash pattern=on 3pt off 3pt] (1.75,-0.97) arc(60.89:201.1:0.72);

%

\draw [line width=2pt, dash pattern=on 3pt off 3pt] (0.49,1.94) arc(159.31:165.81:5.5);

\begin{scope}[yscale=1,xscale=-1]
  \draw [line width=2pt, dash pattern=on 3pt off 3pt] (0.49,1.94) arc(159.31:165.81:5.5);
\end{scope}
\begin{scope}[yscale=-1,xscale=-1]
  \draw [line width=2pt, dash pattern=on 3pt off 3pt] (0.49,1.94) arc(159.31:165.81:5.5);
\end{scope}
\begin{scope}[yscale=-1,xscale=1]
  \draw [line width=2pt, dash pattern=on 3pt off 3pt] (0.49,1.94) arc(159.31:165.81:5.5);
\end{scope}
\fill[draw=blue, fill=blue!40]
(0.29,1.35) arc(82.19:97.81:2.15) arc(14.19:-14.19:5.5) arc(262.19:277.81:2.15) arc(194.19:165.81:5.5);
   \draw [line width=2pt,color=blue]  (0.29, 1.35) arc(82.19:97.81:2.15) node[left] {\small$Q$};
   \draw [line width=2pt,color=blue] (-0.29,-1.35) node[left] {\small$R$} arc(262.19:277.81:2.15);
   \draw [color=blue] (0.27,-1.35) node[right] {\small$S$};
   \draw [color=blue] (0.27,1.35) node[right] {\small$P$};
   \draw [line width=2pt,color=blue] (-0.29,-1.35)  arc(-14.19:14.19:5.5);
   \draw [line width=2pt,color=blue] (0.29,1.35) arc(165.81:194.19:5.5);
\end{tikzpicture}
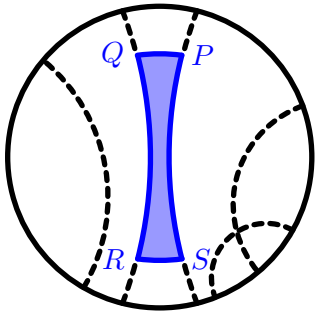
\center{\small{Our domain in the Poincar\'e disc model of $\Ho^2$.} }
\end{center}
 \end{minipage}
\begin{minipage}{0.05\textwidth}  $\quad$
\end{minipage}
  \begin{minipage}{0.45\textwidth}
The picture of the two balls is not entirely accurate in our case because the size of the neck of our domains is not arbitrarily small comparing to the distance from $R$ to $S$ as seen in \eqref{eq:neck-size}. Nevertheless, the presence of a neck of shrinking width allows for the vanishing fundamental gap.
%
  \end{minipage}


\section{Higher dimensions}
\label{sec:higher-dimensions}
\spacing{1.2}

In this section we generalise the above result to higher dimensions. It is a computation in coordinates. We have included minute details for ease of understanding. The coordinates are standard spherical coordinates (unlike the nonstandard coordinates $(r,\varphi)$ in the rest of the article).

\subsection{Coordinates and the Laplace operator in coordinates.}
Let us recall the $n$-dimensional spherical coordinates $(r, \omega_2, \omega_3, \ldots, \omega_n)$:
  \begin{align*}
  x_1 &= r \cos \omega_2\\
  x_2 &= r \sin \omega_2 \cos \omega_3 \\
  x_3 &= r \sin \omega_2 \sin \omega_3 \cos \omega_4 \\
  \vdots & \\
  x_{n-1} &= r \sin \omega_2 \sin \omega_3 \cdots \sin \omega_{n-1} \cos \omega_{n}\\
  x_{n} &= r \sin   \omega_2 \sin \omega_3 \cdots \sin \omega_{n-1} \sin \omega_{n}.
  \end{align*}
  The metric $g_{ij}$ in these coordinates is given by $g_{ij} = 0$ for $i \ne j$ and
  \begin{align*}
    g_{11} & = 1\\
    g_{22} & = r^2 \\
    g_{33} & = r^2 (\sin \omega_2)^2\\
    g_{ii} & = r^2 (\sin \omega_2)^2  \cdots (\sin \omega_{i-1})^2, \quad \text{ for }i=4, \ldots, n
  \end{align*}
  The determinant of the matrix $g_{ij}$ is $g = \det(g_{ij}) = r^{2n-2} (\sin \omega_2)^{2n-4} \cdots (\sin \omega_{n-1})^{2}$.
  The Laplacian (in $\mathbb R^n$) in these coordinates is (where $\omega_1 = r$)
  \[
    \Delta_{\R^n} u = \frac{1}{\sqrt g} \sum_{i,j =1}^n \pd{}{\omega_j} \left( \sqrt g g^{ij} \pd{u}{\omega_i} \right) = \sum_{i=1}^n \frac{g^{ii}}{\sqrt g} \pd{}{\omega_i}\left( \sqrt g \pd{u}{\omega_i} \right),
  \]
  because $g_{ij}$ is diagonal and the entry $g_{ii}$ does not depend on $\omega_i$. Replacing the values of the metric and its inverse in the equation above, we get
  \begin{equation}
    \label{eq:Laplacian-n-spherical}
    \Delta_{\R^n} u = \frac{1}{r^{n-1}} \pd{}{r}\left( r^{n-1}\pd{u}{r} \right) + \sum_{i=2}^n \frac{g^{ii}}{(\sin \omega_i)^{n-i}} \pd{}{\omega_i}\left( (\sin \omega_i)^{n-i} \pd{u}{\omega_i} \right)
  \end{equation}
  \subsection{The domains}
  \label{ssec:domains}
  In order to have a well-defined and computation-suited metric, we center our domain around $\omega_i=\pi/2$. The only coordinate that is not close to zero is $x_n$ and should be the one used for the weight in the hyperbolic half-space model.

  The natural generalization of our domains $\Omega_{\sqrt \mu, L}$ is
  \begin{multline*}
  \Omega_{\sqrt \mu, \delta_2, \ldots, \delta_{n-1}, L} = \{ (r, \omega_2, \ldots, \omega_n)\mid 1<r<e^{\pi/\sqrt \mu},\\
       |\omega_i-\pi/2| <\delta_i, |\omega_n - \pi/2| < L \textrm{ for } i=2, \ldots, n-1\}
     \end{multline*}
     for $L$ and $\delta_i \in (0, \pi/2)$, $i=2,\dots,n-1$. The $\delta_i$'s don't necessarily have to be small but it is easier to picture the domains and convince oneself that the diameter is bounded independently of $\mu$.
      As in the two-dimension case, we will study the fundamental gap for $\mu \to \infty$.  

  The metric on $\mathbb H^n$ we take is $ds_{\mathbb H^n}^2 = \frac{1}{x_n^2} ds_{\R^n}^2$.

  Under a conformal change of metric given by $\tilde g = e^{2f} g$ for a smooth function $f$, the Laplacian changes as
  $
    \Delta_{\tilde g} = e^{-2f} \Delta_g - (n-2) e^{-2f} g^{ij} \frac{\partial f}{\partial x_j}\frac{\partial}{\partial x_i}.
  $ 
   In our case $f= -\ln x_n$ and $\frac{\partial f}{\partial r}=-\frac{1}{r}$, $\frac{\partial f}{\partial\omega_i} = - \frac{\cos \omega_i}{\sin \omega_i}$. The Laplacian is given by
 \[
   \Delta_{\mathbb H^n}= x_n^2 \Delta_{\R^n} + (n-2) x_n^2 \left( \frac{1}{r} \pd{ }{r}+ \frac{1}{r^2} \frac{\cos \omega_2}{\sin \omega_2} \pd{ }{\omega_2}+ \sum_{i=3}^n g^{ii} \frac{\cos \omega_i}{\sin \omega_i} \pd{ }{\omega_i} \right).
 \]
An eigenvalue-eigenfunction pair $\lambda, u$ for $\Omega$ in hyperbolic space satisfies $-\lambda u = \Delta_{\mathbb H^n} u$, which is written in coordinates as
  \begin{multline}
    \label{eq:L-in-coordinates}
    \frac{-\lambda}{r^2 (\sin \omega_2)^2 \cdots(\sin \omega_n)^2 } u =  \frac{1}{r^{n-1}} \pd{}{r}\left( r^{n-1}\pd{u}{r} \right) + (n-2) \frac{1}{r} \pd{u}{r} \\
    + \frac{1}{r^2} \frac{1}{(\sin \omega_2)^{n-2}} \pd{}{\omega_2}\left( (\sin \omega_2)^{n-2} \pd{u}{\omega_2} \right)+ (n-2) \frac{1}{r^2} \frac{\cos \omega_2}{\sin \omega_2} \pd{u}{\omega_2}\\
    + \sum_{i=3}^n \frac{1}{r^2 (\sin \omega_2)^2  \cdots (\sin \omega_{i-1})^2}\left(\frac{1}{(\sin \omega_i)^{n-i}} \pd{}{\omega_i}\left( (\sin \omega_i)^{n-i} \pd{u}{\omega_i} \right) + (n-2) \frac{\cos \omega_i}{\sin \omega_i} \pd{u}{\omega_i}\right).
  \end{multline}

  \subsection{Separation of variables}
  \label{ssec:separation-variables}
  As in the two-dimensional case, let us separate variables. We write
  \[
    u = f(r) \sff_2(\omega_2) \cdots\sff_n(\omega_n)
  \]
  then divide both sides of equation \eqref{eq:L-in-coordinates} by $ u/r^2$ to obtain
  \begin{multline}
    \label{eq:variables-separated}
    \frac{-\lambda}{(\sin \omega_2)^2 \cdots(\sin \omega_n)^2}  =\frac{1}{f r^{n-3}} \pd{}{r}\left( r^{n-1}\pd{f}{r} \right) + (n-2) \frac{r}{f} \pd{f}{r} \\
    +  \frac{1}{\sff_2(\sin \omega_2)^{n-2}} \pd{}{\omega_2}\left( (\sin \omega_2)^{n-2} \pd{\sff_2}{\omega_2} \right)+ (n-2) \frac{\cos \omega_2}{\sff_2(\sin \omega_2)} \pd{\sff_2}{\omega_2}\\
    + \sum_{i=2}^{n} \frac{1}{\sff_i(\sin \omega_2)^2  \cdots (\sin \omega_{i-1})^2}\left(\frac{1}{ (\sin \omega_i)^{n-i}} \pd{}{\omega_i}\left( (\sin \omega_i)^{n-i} \pd{\sff_i}{\omega_i} \right)+(n-2) \frac{\cos \omega_i}{\sin \omega_i} \pd{\sff_i}{\omega_i}\right)
  \end{multline}
  First, we get that the only piece depending on $r$ has to be a constant, say $-\kappa_1$. We rewrite this fact as
  \begin{equation}
    \label{eq:f-equation}
    r^{n-1}\pd{}{r}\left( r^{n-1}\pd{f}{r} \right)+(n-2) r^{n-2} \left( r^{n-1} \pd{f}{r} \right) = - \kappa_1 f r^{2n-4}.
  \end{equation}
For  $n \ge 3$, to solve the ode we change to the variable $t$ so that $\frac{dr}{dt}=-r^{n-1}$, i.e. $t=r^{2-n}/(n-2)$ and \eqref{eq:f-equation} becomes $t^2 \partial^2_{tt} f -  t \partial_t f = - \frac{1}{(n-2)^2} \kappa_1 f $.  Now let $s = \log t$, we have
  \begin{equation}
  \label{eq:f-equation-in-s}
  \partial^2_{ss} f -  2 \partial_s f = - \frac{1}{(n-2)^2} \kappa_1 f,
  \end{equation}
where $s \in ( - (n-2) \frac{\pi}{\sqrt{\mu}} -\log(n-2),\ -\log(n-2))$.

This is a linear second order differential equation whose characteristic polynomial has the roots $r_{1,2}=1\pm \sqrt{1-\frac{\kappa_1}{(n-2)^2}}$. For $f$ to satisfy the Dirichlet boundary conditions, we must have two complex conjugate roots, thus $\frac{\kappa_1}{(n-2)^2}-1=\frac{k^2 \mu}{(n-2)^2}  >0$, for $k$ non-zero integers. Furthermore, as the first eigenfunction $f$ has to be positive and satisfy the Dirichlet boundary conditions,
%
  \begin{equation}
    \label{eq:kappa1}
    \frac{\kappa_1}{(n-2)^2}-1:= \frac{\mu}{(n-2)^2}  >0.
  \end{equation}
  so $f(s)=-e^s \sin \left(\frac{\sqrt{\mu}}{n-2} (s+\log (n-2)\right)$. From here, one can get an explicit solution $f(r)$, but it is not important for the rest of the argument.

  Let us continue with our separation of variables.
   Using  \eqref{eq:f-equation}, we replace the first two terms on the right-hand side of \eqref{eq:variables-separated} by $-\kappa_1$,
  then multiply by $(\sin \omega_2)^2$ to get 
  \begin{multline*}
    \frac{-\lambda}{(\sin \omega_3)^2\cdots(\sin \omega_n)^2}  = -\kappa_1 (\sin \omega_2)^2\\
    +  \frac{1}{\sff_2(\sin \omega_2)^{n-4}} \pd{}{\omega_2}\left( (\sin \omega_2)^{n-2} \pd{\sff_2}{\omega_2} \right)+ (n-2) \frac{(\cos \omega_2)(\sin \omega_2)}{\sff_2} \pd{\sff_2}{\omega_2}\\
    + \sum_{i=2}^{n} \frac{1}{\sff_i(\sin \omega_3)^2 \cdots (\sin \omega_{i-1})^2}\left(\frac{1}{ (\sin \omega_i)^{n-i}} \pd{}{\omega_i}\left( (\sin \omega_i)^{n-i} \pd{\sff_i}{\omega_i} \right)+(n-2) \frac{\cos \omega_i}{\sin \omega_i} \pd{\sff_i}{\omega_i}\right).
  \end{multline*}
  The only piece depending on $\omega_2$ has to be constant, so
  \begin{equation}
    \label{eq:sff2-equation}
    -\kappa_1 (\sin \omega_2)^2 +  \frac{1}{\sff_2(\sin \omega_2)^{n-4}} \pd{}{\omega_2}\left( (\sin \omega_2)^{n-2} \pd{\sff_2}{\omega_2} \right) + (n-2) \frac{(\cos \omega_2)(\sin \omega_2)}{\sff_2} \pd{\sff_2}{\omega_2}= -\kappa_2
  \end{equation}
  For $i=3, \ldots, n-1$, we repeat to get
  \begin{equation}
    \label{eq:sffia-equation}
    -\kappa_{i-1} (\sin \omega_i)^2 +  \frac{1}{\sff_i(\sin \omega_i)^{n-i-2}} \pd{}{\omega_i}\left( (\sin \omega_i)^{n-i} \pd{\sff_i}{\omega_i} \right) + (n-2) \frac{(\cos \omega_i)(\sin \omega_i)}{\sff_i} \pd{\sff_i}{\omega_i}= -\kappa_i,
  \end{equation}
  where each $\kappa_i$ is a constant.
  Until the equation \eqref{eq:variables-separated} becomes
  \begin{equation}
    \label{eq:sffn-equation}
    \frac{-\lambda}{(\sin \omega_n)^2} = - \kappa_{n-1} + \frac{1}{\sff_n}\left(\pd{^2 \sff_n}{\omega_n^2} + (n-2) \frac{\cos \omega_n}{\sin \omega_n} \pd{\sff_n}{\omega_n}\right)
  \end{equation}
  which looks very close to equation \eqref{eq:h-equation}, except for the extra first degree term. Because $\omega_n$ is centered at $\pi/2$, $\sin \omega_n$ should be thought of as $\cos \varphi$. Also note that equations \eqref{eq:sff2-equation} and \eqref{eq:sffn-equation} are contained in the formulation of \eqref{eq:sffia-equation} so it suffices to study \eqref{eq:sffia-equation}.

  \subsection{Solving the ODEs}
  \label{ssec:solving-odes}

Equation \eqref{eq:sffn-equation} can be transformed in such a way that the first-order term is eliminated. It will then be similar to equation \eqref{eq:h-equation}.

  The first step is to multiply \eqref{eq:sffia-equation} by $\sff_i$, move the term with $\kappa_{i-1}$ to the right side and expand the derivative of the product to get
  \begin{equation}
    \label{eq:sffi-equation}
    (\sin \omega_i)^{2} \pd{^2 \sff_i}{\omega_i^2} + (2n-2-i) (\cos \omega_i)(\sin \omega_i) \pd{\sff_i}{\omega_i}= -\kappa_i \sff_i +\kappa_{i-1} (\sin \omega_i)^2 \sff_i,
  \end{equation}
  We look at \eqref{eq:sffi-equation} and want to combine all the derivatives as $\frac{\partial^2}{\partial \omega_i^2} ( (\sin \omega_i)^{\alpha_i} \sff_i)$. In order to do so, the exponent $\alpha_i$ should be half of the constant in front of the term $(\cos \omega_i) (\sin \omega_i) \frac{\partial \sff_i}{\partial \omega_i}$ and therefore $\alpha_i = n-1-\tfrac{i}{2}$. 
  We multiply \eqref{eq:sffi-equation} by $(\sin \omega_i)^{\alpha_i-2}$ and obtain 
\begin{multline}
  \label{eq:sffi-equation2}
  (\sin \omega_i)^{\alpha_i} \pd{^2 \sff_i}{\omega_i^2} + 2 \alpha_i (\cos \omega_i)(\sin \omega_i)^{\alpha_i-1} \pd{\sff_i}{\omega_i}= -\kappa_i (\sin \omega_i)^{\alpha_i-2} \sff_i +\kappa_{i-1} (\sin \omega_i)^{\alpha_2} \sff_i.
\end{multline}
The left-hand side is equal to
  \[
 \pd{^2}{\omega_i^2}\left( (\sin \omega_i)^{\alpha_i} \sff_i \right) - \sff_i \pd{^2}{\omega_i^2}\left( (\sin \omega_i)^{\alpha_i} \right) =
  \pd{^2 \sfh_i}{\omega_i^2} - \alpha_i (\alpha_i-1)\frac{\sfh_i}{(\sin \omega_i)^2} + \alpha_i^2 \sfh_i,
 \]
  where $\sfh_i = \sff_i (\sin \omega_i)^{\alpha_i}$.
  Putting all of this into \eqref{eq:sffi-equation2}, we get
  \begin{equation}
    \label{eq:sfhi-equation}
    \pd{^2 \sfh_i}{\omega_i} +\frac{\kappa_i- \alpha_i (\alpha_i-1)}{(\sin \omega_i)^2} \sfh_i = (\kappa_{i-1} - \alpha_i^2) \sfh_i.
  \end{equation}
  Note that \eqref{eq:sfhi-equation} has the form of \eqref{eq:h-equation} under the change of variable $\varphi = \pi/2 - \omega_i$.
  For $i=n$, the change of variables transforms \eqref{eq:sffn-equation} into
  \begin{equation}
    \label{eq:sfhn-equation}
    \pd{^2 \sfh_n}{\omega_n} +\frac{\lambda- \alpha_n (\alpha_n-1)}{(\sin \omega_n)^2} \sfh_n = (\kappa_{n-1} - \alpha_n^2) \sfh_n,
  \end{equation}
  where $\alpha_n = \frac{n}{2}-1$.

  \subsection{The proof in dimension $n\geq 3$.}
  The first eigenvalue of the Laplace operator on a domain $\Omega_{\sqrt \mu, \delta_2, \ldots, \delta_{n-1}, L}$ can be found in the following way. First, one computes the smallest $\kappa_1$ that is an eigenvalue of \eqref{eq:f-equation}. Then one repeats the process for $i=2, \ldots, n-1$: given $\kappa_{i-1}$, one takes $\kappa_i$ to be the first eigenvalue of \eqref{eq:sfhi-equation}. Finally, with the knowledge of $\kappa_{n-1}$, the first eigenvalue for the Laplacian on the domain $\lambda_1(\Omega_{\sqrt \mu, \delta_2, \ldots, \delta_{n-1}, L})$ is the first eigenvalue of equation \eqref{eq:sfhn-equation}.
  With the same $\kappa_{n-1}$, the second eigenvalue of \eqref{eq:sfhn-equation} is an eigenvalue of the Laplacian on the domain, but not necessarily the second one. Nevertheless, one can use this value as an upper bound for $\lambda_2(\Omega_{\sqrt \mu, \delta_2, \ldots, \delta_{n-1}, L})$. Therefore,  as in the two-dimension case, to prove that $\lambda_2(\Omega_{\sqrt \mu, \delta_2, \ldots, \delta_{n-1}, L}) - \lambda_1(\Omega_{\sqrt \mu, \delta_2, \ldots, \delta_{n-1}, L}) \to 0$, it suffices to show that the difference between the first two eigenvalues of \eqref{eq:sfhn-equation} go to zero. Using Theorem \ref{thm:FG-h-equation}, one just needs the fact that $\kappa_{n-1} \to \infty$ when $\mu \to \infty$.

  As $\mu \to \infty$, the first constant $\kappa_1$ goes to infinity by equation \eqref{eq:kappa1}. Then for each $i=2, \ldots, n-1$, \eqref{eq:sfhi-equation} and Lemma \ref{lem:bound from below2} gives that
  \[
    \kappa_i - \alpha_i(\alpha_i-1) \geq (\cos \delta_i)^2 (\kappa_{i-1} - \alpha_i^2),
  \]
  where the $\alpha_i$ are constant. Consequently, $\kappa_i \to \infty$ for each $i$ and applying Theorem \ref{thm:FG-h-equation} finishes the proof in higher dimensions.

\bibliographystyle{plain}
\bibliography{references}

\end{document}